\documentclass{amsart}
\usepackage{amsmath}
\usepackage{paralist}
\usepackage[pagewise]{lineno}

\usepackage[numbered]{bookmark}
\usepackage{amsmath}
\usepackage{amsfonts}
\usepackage{amsthm} 
\usepackage{amssymb}
\usepackage{cite}
\usepackage{mathrsfs}
\usepackage{changepage} 
\usepackage{mathtools}
\usepackage{mathabx}
\usepackage{verbatim} 
\usepackage{enumerate}

\hypersetup{colorlinks=true, urlcolor=blue, citecolor=red}

\newtheorem{thm}{Theorem}[section]

\newtheorem{lem}[thm]{Lemma}
\newtheorem{prop}[thm]{Proposition}
\newtheorem{cor}[thm]{Corollary}
\DeclareMathOperator{\id}{id}

\newcommand{\R}{\mathbb{R}}
\newcommand{\Z}{\mathbb{Z}} 
\newcommand{\N}{\mathbb{N}}
\newcommand{\C}{\mathbb{C}} 
 
\newcommand{\td}{\mathbb{T}^d}
\newcommand{\torus}{\mathbb{T}}
\newcommand{\T}{\mathbb{T}}
\newcommand{\TR}{\mathbb{T}^d\times\mathbb{R}^d}
\newcommand{\rd}{\mathbb{R}^d} 
\newcommand{\oir}{\mathcal{O}^\infty(I)}
\newcommand{\dioph}{\Omega^\tau_\gamma}
\newcommand{\hone}{\hspace{1cm}}

\newcommand{\DC}{\textup{DC}(\gamma, \tau)}
\newcommand{\DCd}{\textup{DC}_d(\gamma, \tau)}

\newcommand{\SE}[3]{\mathcal{SE}^d_{r, s}}
\newcommand{\Function}[5]{\begin{array}{cccc} #1 : & #2 & \rightarrow & #3 \\ & #4 & \mapsto & #5 \end{array}}
\newcommand{\RBd}[1]{\mathbb{B}_{#1}^d}
\newcommand{\CBd}[1]{\mathbb{D}_{#1}^d}
\newcommand{\RDom}[1]{\mathbb{T}^d \times \RBd{#1}}
\newcommand{\CDom}[2]{\mathbb{T}^d_{#1} \times \CBd{#2}}
\newcommand{\std}{\textup{std}}

\begin{document}

\title[Quasi-analytic props. of the KAM curve]{Quasi-analytic properties of the KAM curve}
\author[Frank Trujillo]{}
\email{frank.trujillo@imj-prg.fr}
\maketitle

\centerline{\scshape Frank Trujillo}
\medskip
{\footnotesize
 \centerline{IMJ-PRG, Universit\'e de Paris}
   \centerline{Paris, France}
}

\bigskip

\begin{abstract}
Classical KAM theory guarantees the existence of a positive measure set of invariant tori for sufficiently smooth non-degenerate near-integrable systems. When seen as a function of the frequency this invariant collection of tori is called the KAM curve of the system. Restricted to analytic regularity, we obtain strong quasi-analyticity properties for these objects. In particular, we prove that KAM curves completely characterize the underlying systems. We also show some of the dynamical implications on systems whose KAM curves share certain common features.  
\end{abstract}

\maketitle

\section{Introduction} 

\subsection{Motivation}
\label{Motivation}

Consider the analytic standard family of symplectic transformations given by
\begin{equation*}
\label{standard_family}
\Function{F_{\epsilon,\varphi}}{\torus \times \R}{\torus \times \R}{ (\theta, I)}{(\theta+ I +\epsilon \varphi(\theta), I + \epsilon \varphi(\theta))}
\end{equation*}
where $\epsilon \in \R$ and $\varphi \in C^\omega(\T, \R)$ has zero mean value. For $\varphi$ fixed and for any $\gamma, \tau > 0$ the classical KAM theorem guarantees that for any $\epsilon$ verifying
\[ |\epsilon| < \epsilon_0(\varphi, \gamma, \tau), \]
where $\epsilon_0(\varphi, \gamma, \tau)$ is a positive constant given by the theorem, there exist a collection $\{T_\omega\}_{\omega \in \Omega}$ of invariant curves for the mapping $F_{\epsilon, \varphi}$ whose rotation numbers are in bijection with the set $\DC$ of \textit{Diophantine numbers of type} $(\gamma, \tau)$, namely, with the set of real numbers $\omega$ verifying
\[ |q\omega - p| \geq \frac{\gamma}{|q|^{\tau + 1}}, \hskip0.5cm \text{ for all } p, q \in \Z, \, q \neq 0.\]
Recall that for any $\gamma, \tau > 0$, $\DC$ is a Cantor set of positive Lebesgue measure in $\R$. Since the invariant curves $\{T_\omega\}_{\omega \in \Omega}$ are actually graphs of mappings in $C^\omega(\T, \R)$ this collection can be encoded as a function
\[T_{F_{\epsilon,\varphi}} : \DC \rightarrow C^{\omega}(\T, \R)\]
depending only on the rotation number of the invariant curves. Following \cite{carminati_there_2014}, we call $T_{F_{\epsilon, \varphi}}$ the \textit{KAM curve} associated to $F_{\epsilon, \varphi}$. This construction can be made in a much more general perturbative setting but, for the sake of clarity, we postpone its formal definition to Section \ref{results} and momentarily restrict ourselves to the analytic standard family. 

Notice that the KAM curve $T_{F_{\epsilon, \varphi}}$ is only defined over a Cantor set and thus its differentiability properties are better understood in terms of \textit{Whitney smoothness}, that is, wether or not the function admits a smooth extension to an open neighbourhood of its initial domain. For this particular example, it follows from the results of Z. Shang \cite{shang_note_2000} that $T_{F_{\epsilon, \varphi}}$ is in fact $C^\infty$ smooth in the sense of Whitney. We point out that Shang's result holds in a much more general setting and that analogous results had been previously stablished by V. Lazutkin \cite{lazutkin_existence_1973} in the context of convex billiards and by J. P\"oschel \cite{poschel_integrability_1982} in the Hamiltonian setting (see Section \ref{results} for the definition of KAM curves of Hamiltonians). 
 
In \cite{carminati_there_2014} C. Carminati et al. showed that KAM curves for the analytic standard family $F_{\epsilon, \varphi}$ are not only smooth in the sense of Whitney but they also admit, in a natural way, a unique extension to certain space of holomorphic functions. One of the main interests of these extensions comes from the quasi-analytic properties of such spaces. As a particular application of these properties, one can deduce the following.

\begin{prop}
\label{standard_family_prop}
Let $\varphi, \psi \in C^\omega(\torus, \R)$ and 
$ \epsilon < \min\{ \epsilon_0(\varphi, \gamma, \tau), \epsilon_0(\psi, \gamma, \tau)\}$ where $\epsilon_0$ denotes the constants given by the KAM theorem. If the two KAM curves $T_{F_{\epsilon,\varphi}}$ and $T_{F_{\epsilon,\psi}}$ coincide on a set $\Gamma \subset \DC$ of positive measure then $T_{F_{\epsilon,\varphi}} = T_{F_{\epsilon,\psi}}$. 
\end{prop}

Quoting the authors in \cite{carminati_there_2014}, \textit{``the knowledge of parametrizations of invariant tori on a set of positive measure of rotation numbers is sufficient to determine all the parametrized KAM curves''}. Let us point out that in general the KAM curves are not analytic since this would imply the complete \textit{integrability} of the system, that is, the space would be completely foliated by invariant tori. Nevertheless, as shown by Proposition \ref{standard_family_prop}, the KAM curves do preserve (in a weak sense) some of the classical properties of analyticity. In \cite{carminati_there_2014} the authors suggest that an analogous of the quasi-analytic extension of the KAM curves and its uniqueness properties should exist for general \textit{near-integrable} systems, i.e. for perturbations of completely integrable systems, in any dimension. 

In this paper we show part of their intuition correct by proving that for general near-integrable systems the associated KAM curves do exhibit strong quasi-analyticity properties. We explore how and to what extent some of the properties of the KAM curve $T_F$ characterize a general near-integrable analytic system $F$. We will tackle this question in both the discrete (exact symplectic transformations) and continuous (Hamiltonian flows) cases. The techniques we employ are different from those in \cite{carminati_there_2014} and do not make use of the aforementioned quasi-analytic extension. Nevertheless, this approach will allow us to conclude stronger uniqueness properties. 

\subsection{The generalized standard family}

Before going any deeper in the discussion we would like to stress the need to deal with general near-integrable systems when considering only uniqueness properties of the KAM curve and not the extensions proposed in \cite{carminati_there_2014}. In fact, a much stronger conclusion than that of Proposition \ref{standard_family_prop} holds for the \textit{generalized standard family} of exact symplectic maps on the $d$-dimensional cylinder $\T^d \times \R^d$ which is defined as 
\begin{equation*}
\label{generalized_standard_family}
 S_\varphi(\theta, I) = (\theta + I + \varphi(\theta), I + \varphi(\theta)),
 \end{equation*}
where $\varphi \in C^1(\T^d, \R^d)$ is of the form $\varphi = \nabla V$ for some $V \in C^2(\td, \R).$
\begin{prop}
\label{generalized_standard_family_general}
Let $\varphi , \psi \in C^1 (\td,\R^d)$ and $\gamma \in C(\td, \R)$. Suppose that the graph of $\gamma$ is invariant under $S_\varphi$ and $S_\psi$. Then $ \varphi = \psi$. 
\end{prop} 
\begin{proof} Let $\pi _1 , \pi _2$ denote the projections of $\TR $ onto $\td$ and $\rd$ respectively. Define $\overline{\gamma}: \td \rightarrow \TR$ and $g : \td \rightarrow \td$ as 
\[\overline{\gamma}(\theta)= (\theta, \gamma(\theta)), \hskip1cm g(\theta) = (\pi _1 \circ S_\varphi \circ \overline{\gamma})(\theta).\]
The function $g$ is clearly a torus homeomorphism. By the invariance of  the graph of $\gamma$ it follows that   
\[S_\varphi \circ \overline{\gamma} = \overline{\gamma} \circ g,\]
which implies 
\[ \id_{\td}+ \gamma + \varphi = g, \hskip1cm  \gamma + \varphi = \gamma \circ g.\]
Then $g^{-1} = \id - \gamma$ and thus $\varphi$ is uniquely defined by $\gamma$. Since the same holds if we replace $\varphi$ by $\psi$ in the previous argument it is clear that $\varphi = \psi$.
\end{proof}

\subsection{Description of the results}

After recalling in Section \ref{sc: symplectic_geometry} some of the basic concepts from symplectic geometry that will be used throughout this paper we introduce the notion of $C^\infty$-\textit{uniqueness set}, which is simply the analogous in higher dimensions of what a set with an accumulation point is for analytic functions of one variable (see Section \ref{sc: uniqueness_set} for a precise definition and some of its properties).

The formal definition of the KAM curve for general near-integrable systems is given at the beginning of Sections \ref{sc: mappings} and \ref{sc: flows} for exact symplectic transformations and Hamiltonian flows respectively.  We show in Theorems \ref{main_symplectic_maps} and \ref{main_hamiltonians}, for discrete and continuous analytic systems respectively, that whenever two KAM curves coincide on a $C^\infty$-uniqueness set not only the KAM curves but the underlying systems themselves must be equal. Let us point out that being a $C^\infty$-uniqueness set is a much weaker condition than having positive Lebesgue measure. In fact, Proposition \ref{ExampleUniquenessSet} provides examples of countable  $C^\infty$-uniqueness sets in $\R^2$.

Also in Theorems \ref{main_symplectic_maps} and \ref{main_hamiltonians} and under the weaker assumption that the image of the KAM curves of two near-integrable systems intersect on a $C^\infty$-uniqueness set, i.e. the systems share many of the invariant tori but no assumption on the restricted dynamics is made, we show that the systems must necessarily commute.  

Finally, restricted to Hamiltonian systems, we refine the previous result by studying the implications of different assumptions on the restricted dynamics of the common invariant tori. If the rotation vectors on these tori are collinear, the last assertion in Theorem \ref{main_hamiltonians} shows that one of the flows is actually a time reparametrization of the other. On the other hand, if these rotation vectors are not collinear and the system has two degrees of freedom then both Hamiltonians can be simultaneously conjugated by an analytic symplectic transformation to completely integrable systems. This will be a consequence of Theorem \ref{SimConjug}.

\section{Preliminaries}
\label{sc: preliminaries}

\subsection{Symplectic geometry}
\label{sc: symplectic_geometry}

Let us recall some of the rudiments of symplectic geometry. For proofs and a complete introduction to the subject we refer the reader to \cite{cannas_da_silva_lectures_2001}. 

A smooth manifold $M$ of dimension $2d$ endowed with a closed, non-degenerated $2$-form $\omega$ is called a \textit{symplectic manifold}. We will sometimes explicit dimension of $M$ by writing $M^{2d}$. For any open set $U \subset M$ the pair $(U, \omega_M\mid_U)$ is a symplectic manifold. A submanifold $L \subset M$ is said to be \textit{Lagrangian} if the restriction of the symplectic form to $L$ is equal to zero and $\dim(L) = \frac{1}{2}\dim(M)$. 

A smooth function on $M$ is called a \textit{Hamiltonian}. Every Hamiltonian $H$ defines a unique smooth vector field $X_H$ obeying
\[ i_{X_H}\omega = dH,\]
where $i_{X_H}\omega$ is the $1-$ form on $M$ given by 
\[ i_{X_H}\omega(p)(v_p) = \omega(X_H(p),v_p).\]
We say that $X_H$ is the \textit{Hamiltonian vector field} of $H$ and we denote its flow by $\Phi_H^t.$ The \textit{Poisson bracket} $\{H, L\}$ of two Hamiltonians is defined by 
\[\{ H, L\} = X_H(L). \]
Whenever $\{H, L\} = 0$ we say that the functions $H, L$ are in \textit{involution}. A diffeomorphism $\Psi : N \rightarrow M$ between two symplectic manifolds $(N,\omega_N)$ and $(M, \omega_M)$ is said to be \textit{symplectic} if 
\[ \psi^* (\omega_M) = \omega_N,\]
where $\psi^* (\omega_M)$ denotes the pull-back of $\omega_M$ by $\psi$.  A symplectic manifold $(M,\omega)$ is said to be \textit{exact} if the form $\omega$ is exact, that is, if there exist a $1$-form $\alpha$ such that 
\[\omega = d\alpha. \]
A diffeomorphism $\psi : N \rightarrow M$ between two exact symplectic manifolds $(N, d\alpha_N)$ and $(M,d\alpha_M)$ is said to be \textit{exact symplectic} if 
\[ \psi^*(\alpha_M) = \alpha_N.\]
In particular, every exact symplectic mapping is symplectic.  In the following proposition we recall some of the properties of Hamiltonian vector fields.

\begin{prop}
\label{basics_Hamiltonians}
Let $(M,\omega_M)$, $(N, \omega_N)$ symplectic manifolds, $\Sigma : N \rightarrow M$ a symplectomorphism and $H \in C^\infty(M, \R)$. Denote $\mathcal{H} = H \circ \Sigma$. The following holds:
\begin{enumerate}
\item H is constant along the solutions of $X_H$. 
\item For all $t_0 \in \R$ for which the flow $\Phi^{t_0}_H$ is well defined the mapping $\Psi = \Phi^{t_0}_H$ is symplectic. Furthermore, if $M$ is exact then $\Psi$ is exact symplectic. 
\item For all $t \in \R$ for which the flows $\Phi^t_H, \Phi^t_{\mathcal{H}}$ are well defined
\[ D\Sigma \cdot X_{\mathcal{H}} = X_H \circ \Sigma, \hspace{1cm} \Sigma \circ \Phi^t_\mathcal{H} = \Phi^t_H \circ \Sigma. \] 
\item For all $L \in C^\infty(M)$
\[ \{ H\circ \Sigma, L \circ \Sigma\} = \{ H,L\} \circ \Sigma, \hspace{1cm} [X_H,X_L] = X_{\{ H,L\}},\]
where $[\cdot , \cdot]$ denotes the Lie bracket.  In particular, the flows $\Phi^t_H,$ $\Phi^t_L$ commute if and only if $\{ H,L\} = 0.$ 
\end{enumerate}
\end{prop}

A system on a symplectic manifold $M^{2d}$ (Hamiltonian flow or symplectomorphism) is said to be \textit{integrable} if there exist functions $f_1, f_2, \dots ,f_d \in C^\infty (M, \R)$ such that:
\begin{enumerate}
\item $f_1, \dots, f_d$ are invariant by the system,
\item $f_1, f_2, \dots ,f_d$ are generically independent, i.e., $df_1, \dots, df_n$ are linearly independent almost everywhere,
\item $f_i, f_j$ are in involution (i.e. $\{f_i, f_j\} = 0$) for every $i, j=1,\dots,d.$
\end{enumerate}

Functions invariant by the system are called \textit{integrals} of the system. For integrable systems and under fairly general conditions the Arnold-Liouville-Mineur theorem assures that we can locally describe the dynamics of the system in a simplified set of coordinates known as \textit{angle-action coordinates}. 
\begin{thm}[Arnold-Liouville-Mineur]
\label{arnold_liouville}
Let $(M^{2d},\omega)$ be a symplectic manifold and let $f_1,f_2,\dots,f_d \in C^\infty(M, \R)$ be $d$ generically independent functions. Consider $F = (f_1,\dots,f_d)$ and suppose $0 \in \R^d$ is a regular value of $F$ and $M_0$ is a compact connected component of $F^{-1}(0)$. Then there exists an open neighbourhood $U$ of $M_0$ and a symplectomorphism 
\[\psi : U \rightarrow \mathbb{T}^d \times B,\]
where $B$ is an open ball centred at the origin and $\mathbb{T}^d \times B$ is endowed with the canonical symplectic form $d\theta \wedge dI$, such that $F \circ \psi ^{-1}$ depends only on $I$.
\end{thm}

The new coordinates $\theta_i$ and $I_i$ are called \textit{angle} and \textit{action} coordinates respectively. For a proof of this theorem we refer the reader to \cite{arnold_mathematical_2007}. 

By the previous theorem, for any integrable Hamiltonian $H \in C^\infty(M, \R)$ there exists (locally) a symplectic change of coordinates $\psi : U \subset M \rightarrow \mathbb{T}^d \times B$ such that the Hamiltonian flow associated to $h = H \circ \Psi$ is given by 
 \begin{equation}
 \label{integrable_Hamiltonian_form}
 (t, \theta, I) \mapsto (\theta + t \nabla h (I), I).
 \end{equation}
Similarly, for any smooth symplectomorphism $\Sigma: M \rightarrow M$ there exists (locally) a symplectic change of coordinates $\psi : U \subset M \rightarrow \mathbb{T}^d \times B$ such that $\psi^{-1}(\T^d \times \{I_0\})$ is invariant for all $I_0 \in B$. Thus 
 \[ \psi \circ \Sigma \circ \psi^{-1}(\theta, I) = (g(\theta, I), I),\]
for some smooth function $g$. Since the RHS transformation in the previous equation is symplectic, $g$ must be of the form 
\[g(\theta, I) = \theta + \sigma(I)\]
 for some smooth function $\sigma$. Therefore 
 \begin{equation}
 \label{integrable_map_form}
 \psi \circ \Sigma \circ \psi^{-1}(\theta, I) = (\theta + \sigma(I), I).
 \end{equation}
 Since the symplectic change of coordinates $\psi$ establishes a conjugacy with the initial system, the Arnold-Liouville-Mineur theorem asserts that every integrable system (Hamiltonian flow or symplectomorphism) is locally equivalent to a system of the form (\ref{integrable_Hamiltonian_form}) or (\ref{integrable_map_form}) defined over $\T^d \times B \subset \T^d \times \R^d$ and endowed with its canonical symplectic form. 
 
 Since in this work we will only be interested in perturbations of integrable systems and in the persistence of local phenomena, in the following we only consider only systems defined over $\T^d \times B \subset \T^d \times \R^d$, endowed with its canonical symplectic form, and we will refer to \textit{integrability} of the system as wether or not the system can be symplectically conjugated to the form (\ref{integrable_Hamiltonian_form}) or (\ref{integrable_map_form}). 
 
 \subsection{Uniqueness sets}
\label{sc: uniqueness_set}

Let $M$ be a manifold and $p \in M$. We say that a set $K \subset M$ is a $C^\infty$-\textit{uniqueness set at $p$} if for all $C^\infty$ functions defined on an open connected neighbourhood of $K$ such that \[ f\mid_K = 0\] $f$ and its derivatives of all orders are equal to zero at $p$, that is, $f$ is \textit{flat at $p$}. Notice that in dimension one a $C^\infty$-uniqueness set at $p$ is simply a set that accumulates at $p$. 

The next proposition provides some of the properties of $C^\infty$-uniqueness sets that will be used along the paper. 
\begin{prop}
\label{UniSet}
Let $M_1, M_2$ be smooth manifolds. Let $p_i \in K_i \subset M_i$, $i = 1,2$. The following holds:
\begin{enumerate}
\item If $Leb(K_1) > 0$ then $K_1$ is a $C^\infty$-uniqueness set for almost all $p \in K_1$.
\item If $K_1,K_2$ are $C^\infty$-uniqueness set at $p_1$ and $p_2$ respectively then $K_1 \times K_2 \subset M_1 \times M_2$ is a $C^\infty$-uniqueness set at $(p_1,p_2).$
\item  Let $U$ be an open neighbourhood of $K_1$. If $h: U\subset M_1 \rightarrow M_2$ is a smooth diffeomorphism to its image then $h(K_1)$ is a $C^\infty$-uniqueness set at $h(p_1)$.
\end{enumerate}
\end{prop}
\begin{proof} We suppose WLOG that $M_1 = \R^m$, $M_2 = \R^n$. We denote by $\mu$ the Lebesgue measure on $M_1$. 

$\mathit{1.}$ For a full measure set $K \subset K_1$ 
\[ \lim_{r \rightarrow 0}\dfrac{\mu(B_r(x) \cap K_1)}{\mu(B_r(x))} = 1, \]
for all $x \in K$. This easily implies that $\nabla f (x) = 0$ for all $x \in K$. A simple inductive argument shows the assertion. 

$\mathit{2.}$ Let $f \in C^\infty(\R^m \times \R^n, \R)$ and $(\alpha,\beta) \in \N^m \times \N^n$. Then 
\[ \partial^\alpha f (p_1,y) = 0. \]
for all $y \in K_2$. Hence 
\[ \partial^\beta \partial^\alpha f (p_1,p_2) = 0, \]
which proves the assertion. 

$\mathit{3.}$ Notice that for all $f \in C^\infty(\R^n)$, $f$ is flat at $h(p_1)$ if and only if $f \circ h^{-1}$ is flat at $p_1$.
\end{proof}

Positive measure is a sufficient condition for a set to be of uniqueness but it is far from being necessary. In the following given a set $A \subset \R^n$ we denote by $A'$ the set of accumulation points of $K.$

\begin{prop}
\label{ExampleUniquenessSet}
Let $K \subset \R^2$ such that $K' = \{0\}$ and denote
\[ A:= \left\{ \frac{p}{| p |} \,\middle|\, p \in K \,\setminus\, \{ 0 \}\right\}.\]
If $A'$ is infinite then $K$ is a $C^\infty$-uniqueness set at $0$. 
\end{prop}
\begin{proof} Let $f \in C^\infty(\R^2, \R)$ such that $f \mid_K = 0$ and suppose that $f$ is not flat at $0$. Then there exist $N >1$, $N \geq k_0 \geq 0$ and $C > 0$ such that 
\[ f(x,y) = \dfrac{1}{N!}\sum_{k = k_0}^N \partial_x^{k}\partial_y^{N-kf}(0) x^{k}y^{N-k} + F_N(x,y)\]
with \[a := \partial_x^{k_0}\partial_y^{N-k_0}f(0) \neq 0\] and 
\[ |F_N(x,y)| \leq C| (x,y) |^{N+1} \]
for all $(x,y) \in B_1(0)$. Up to composition with a rotation we can suppose WLOG that for all $m \in \N$ there exist sequences $\big(x_n^{(m)}\big)_{n \in \N}, \big(y_n^{(m)}\big)_{n \in \N} $ in $K$ such that
\[ |x_n^{(m)}|, |y_n^{(m)}| \xrightarrow{n \rightarrow \infty} 0, \hspace{1cm} \dfrac{|x_n^{(m)}|}{|y_n^{(m)}|} \xrightarrow{n \rightarrow \infty} \lambda_m, \hspace{1cm} \lambda_m \searrow 0.\]
Then 
\begin{align*}
0 & = \left|\dfrac{N! f\big(x_n^{(m)},y_n^{(m)}\big)}{\big(y_n^{(m)}\big)^N} \right|\\
 & \geq \left| a\left( \dfrac{x_n^{(m)}}{y_n^{(m)}} \right)^{k_0}\right| - \left| \left( \dfrac{x_n^{(m)}}{y_n^{(m)}} \right)^{k_0+1}\sum_{k > k_0}^N \left( \dfrac{x_n^{(m)}}{y_n^{(m)}} \right)^{N - k_0 - 1}\partial_x^{N-k}\partial_y^kf(0)\right| \\
 & - CN! \left| \big(x_n^{(m)},y_n^{(m)}\big) \right| \left|\left(\dfrac{x_n^{(m)}}{y_n^{(m)}},1\right)\right|^{N} \\
 & \xrightarrow{n \to \infty} |a|\lambda_m^{k_0} - \lambda_m^{k_0+1} \sum_{k > k_0}^N \lambda_m^{N- k_0 -1}\left|\partial_x^{N-k}\partial_y^kf(0)\right|.
\end{align*}
As the last expression is positive for $m$ sufficiently large this is a contradiction. \end{proof}

Similar conditions show the existence of countable $C^\infty$-uniqueness sets in any dimension. We finish this section with a simple remark on the composition of analytic maps.
\begin{lem} 
\label{FlatAnalytic}
Let $F,$ $G,$ $\Phi,$ $\Psi$ be $C^\infty$ diffeomorphisms defined on open, connected neighbourhoods of $\td \times \{0\} \subset \TR$. Suppose that $F,$ $G$ are analytic and that the compositions
 \[\overline{F} := \Psi \circ F \circ \Phi, \hspace{1cm} \overline{G} := \Psi \circ G \circ \Phi,\]
 are well defined. If $\overline{F} = \overline{G} + \oir $ then $F = G$.
\end{lem}
\begin{proof} Let $\overline{F} = \overline{G} + H.$ Up to consider appropriate liftings of $F,$ $G,$ $\Phi,$ $\Psi$, which as an abuse of notation we will denote by the same letters,  we can suppose WLOG that $\overline{F}, \overline{G}, H : \R^d \times U \rightarrow \R^d \times \R^d$ for some open neighbourhood $U \subset \R^d$ and that $H = \oir.$  Then for any $(\theta, I) \in \R^d \times U$ 
\begin{align*}
 F(\theta, I) & = \Psi^{-1}  \circ (\Psi \circ G + H \circ \Phi^{-1})(\theta, I) \\
 & = G(\theta, I) + \int_0^1 D\Psi^{-1}(\Psi \circ G (\theta, I)+ t H \circ \Phi^{-1}(\theta, I)) H \circ \Phi^{-1}(\theta, I) dt. 
\end{align*}
Since $H \circ \Phi^{-1}$ is flat at $\Phi(\R^d \times \{ 0 \})$ and $F,$ $G$ are analytic it follows that $F = G$.
\end{proof}
 
\section{The KAM curve}
\label{results}

Let us start this section by introducing some of the notations that will appear in the rest of the paper. 

\subsection{Notations}

Given $z \in \C$ we denote its modulus by $|z|$. For $z \in \C^d$ we denote 
\[| z| = \sqrt{|z_1|^2 + \dots + |z_d|^2}.\] 
Given $s > 0$ and $d \in \N$ denote
\begin{gather*}
\RBd{s} = \left\{ z \in \R^d \,\big|\, |z| < s \right\}, \hskip0.5cm \CBd{s} =  \left\{ z \in \C^d \,\big|\, |z| < s \right\}, \hskip0.5cm \td_s = \left( \mathbb{D}^1_{s} / \Z\right)^d .
\end{gather*}
For any $f: U \subset \C^d \rightarrow \C^n$ we denote its sup-norm by 
\[ \| f\|_U = \sup_{z \in U} | f(z) |.\]
Let $K \subset \R^d$ closed. Given $f: K \subset \R^d \rightarrow \R^n$ we say that $f$ is \textit{$C^\infty$ smooth in the sense of Whitney} if there exist an open neighbourhood $U$ of $K$ and a function $\widetilde{f} \in C^\infty(U, \R^n)$ such that $\widetilde{f}\mid_K = f.$ We denote the set of $C^\infty$ Whitney-smooth functions on $K$ and taking values on $\R^n$ by $C^\infty(K, \R^n)$. In other words
\[ C^\infty(K, \R^n) = \bigcup_{K \subset U} C^\infty(U, \R^n),\]
where the union is taken over all the open neighbourhoods $U \subset \R^d$ of $K$. Similarly, we denote by $\textup{Symp}^\infty(\T^d \times K, \T^d \times \R^d)$ the set of $C^\infty$ Whitney-smooth functions on $\td \times K$ taking values in $\td \times \R^d$ that admit a smooth extension to a symplectic embedding of some open neighbourhood $\td \times U \subset \td \times \R^d$  of $\T^d \times K$ into $\T^d \times \R^d$. Given $z \in \C^d$ we denote
\[\|z\| = \min_{k \in \Z^d} |z - k|.\]
Given $\gamma, \tau > 0$ we say that $\omega \in \R^d$ is \textit{Diophantine of type} $(\gamma, \tau)$ if it satisfies
\[ \| \langle \omega, k \rangle \| \geq \dfrac{\gamma}{|k|^{d + \tau}} \hspace{1cm} \text{ for all } k \in \Z^d \,\setminus\, \{ 0 \}.\] 
We denote the set of Diophantine numbers of type $(\gamma, \tau)$ by $\DCd$. Recall that for any $\gamma, \tau > 0$ the set $\DCd$ has positive Lebesgue measure and for any bounded open set $\Omega \subset \R^d$
\[ \textup{Leb}(\DCd \cap \Omega) \xrightarrow{\gamma \to 0} \textup{Leb}(\Omega). \]
For any $\gamma, \tau > 0$ and any open bounded set $\Omega \subset \rd$ we define 
\[ \Omega^\tau_{\gamma} := \{ \omega \in \Omega \mid \omega \in \DCd, \, d(\omega,\partial \Omega) > \gamma \}.\]
Given $f : \T^d \rightarrow \C$ we will denote its average over $\T^d$ by $[f]$.

\subsection{Exact symplectic maps}
\label{sc: mappings}

As mentioned in the introduction, the KAM curve associated to a sufficiently small perturbation of a non-degenerate integrable system consists of the collection of invariant tori given by the KAM theorem when encoded as a function of the Diophantine frequencies. To formalize the definition we state a simplified version of the KAM theorem for exact symplectic transformations found in \cite{shang_note_2000}. 

To simplify the exposition, let us start by introducing a suitable space of transformations. As we will be interested in analytic symplectic transformations defined on domains of the form $\RDom{s}$ it is useful to consider the space of symplectic transformations defined on a fixed complex neighbourhood of $\RDom{s}$. Given $r, s > 0$, $d \in \N$ we define the space of \textit{exact symplectic embeddings} of $\CDom{r}{s}$ as
\[ \SE{d}{r}{s} = \left\{ F: \CDom{r}{s} \rightarrow \T^d_\infty \times \C^d \ \left| \begin{array}{l} F \textup{ is real analytic and } F\mid_{\RDom{s}} \\ \textup{is an exact symplectic embedding.}\end{array}\right.\right\}\]
We endow $\SE{d}{r}{s}$ with the $C^0$-topology. 

\begin{thm}[KAM theorem for symplectomorphisms]
\label{KAM_symplectic} 
Let $r, s > 0$ and $d \in \N$. Suppose $S_0 : \CBd{s} \rightarrow \C$ real analytic such that $\partial_I S_0\mid_{\CBd{s}}$ is a diffeomorphism onto its image and 
\[ \|\partial^2_IS_0\|_{\CBd{s}}, \|\partial^2_IS_0^{-1}\|_{\CBd{s}} < +\infty.\]
Denote by $F_0$ the associated exact symplectic map in $\SE{d}{r}{s}$ given by
\[F_0(\theta, I) = (\theta + \partial_I S_0(I), I)\]
and let 
\[\Omega = \partial_I S_0(\RBd{s}).\]
Given $\gamma, \tau > 0$ there exists an open neighbourhood $\mathcal{U}_{\gamma, \tau}$ of $F_0$ in $\SE{d}{r}{s}$ such that for any $F \in \mathcal{U}_{\gamma, \tau}$ there exist a Cantor set $K \subset \RBd{s}$ and Whitney smooth functions 
\[ S \in C^\infty(K, \R), \hone \Sigma \in \textup{Symp}^\infty(\T^d \times K, \T^d \times \R^d),\]
such that $\partial_I S : K \rightarrow \Omega^\tau_\gamma$ is a bijection and
 \[ \Sigma ^{-1} \circ F \circ \Sigma \mid_{\td \times K}(\theta, I) =(\theta + \partial_IS(I), I).\]
\end{thm} 

Using the notations in the previous theorem, it follows that for any $F \in \mathcal{U}_{\gamma, \tau}$ and for all $\omega \in \Omega^\tau_\gamma$ the graph of
\[ \Function{u_\omega}{\td}{\R^d}{\theta}{\Sigma (\theta, \partial_IS^{-1}(\omega))}\]
defines an invariant Lagrangian torus $T_\omega$ for $F$ whose restricted dynamics is conjugated to a translation by $\omega$. We encode the collection of invariant tori $\{T_{\omega}\}_{\omega \in \Omega^\tau_\gamma}$ in the \textit{KAM curve} $T_F$ defined as the Whitney smooth function
\[ \begin{array}{cccc} 
T_F : & \Omega^\tau_{\gamma} & \rightarrow & C^{\infty} (\mathbb{T}^d) \\ 
& \omega & \mapsto & u_\omega
\end{array}. \] 
Following \cite{eliasson_around_2015} we say that a smooth Lagrangian invariant torus whose restricted dynamics are smoothly conjugated to a translation by a Diophantine vector $\omega$ is a \textit{KAM torus} with \textit{rotation vector} $\omega$. 

We can now state the main result of this section.

\begin{thm}
\label{main_symplectic_maps}
Let $F_0,$ $\Omega$, $\mathcal{U}_{\gamma, \tau}$ as in Theorem \ref{KAM_symplectic}. Let $F \in \mathcal{U}_{\gamma, \tau}$ and denote by $T_F$ the associated KAM curve. Suppose 
\[G: \RDom{s} \rightarrow \td \times \R^d\]
 is an analytic exact symplectic embedding and $\Gamma \subset \dioph$ is a $C^\infty$-uniqueness set at $\omega_0 \in \Gamma$. If for all $\omega \in \Gamma$ the graph of $T_F(\omega)$ defines an invariant torus $T_\omega$ for $G$ then the following holds:
 \begin{enumerate}
 \item $T_{\omega_0}$ is an invariant KAM torus for $G$. 
 \item $F$ and $G$ commute on a neighbourhood of $T_{\omega_0}$.
 \item If $T_\omega$ is a KAM torus for $G$ with rotation vector $\omega$ for all $\omega \in \Gamma$ then $F = G$.
\end{enumerate}
 \end{thm}
 \begin{proof}
 Let $\Sigma$, $S$ as in Theorem \ref{KAM_symplectic} when applied to $F$. We can assume WLOG that $\Sigma$, $S$ are well-defined smooth functions on open neighbourhoods of $\td \times K$ and $K$ respectively.  Denote 
\[ \overline{F} := \Sigma ^{-1} \circ F \circ \Sigma, \hspace{1cm} \overline{G} := \Sigma ^{-1} \circ G \circ \Sigma.\]
Let us write $\overline{F}, \overline{G}$ as
\[\overline{F} (\theta,I) = (\theta+S(I) + f_1(\theta,I),I + f_2(\theta,I)),\]
\[\overline{G}(\theta, I) = (\theta + g_1(\theta,I),I + g_2(\theta,I)),\]
where $f_1, f_2, g_1, g_2$ are $C^\infty$ functions. Then for all $ (\theta, I) \in \td \times \partial_I S^{-1}(\Gamma)$
 \[ f_1(\theta,I) = 0, \hspace{1cm} f_2(\theta,I) = 0 = g_2(\theta,I). \]
Let us assume WLOG that $\partial_I S^{-1}(\omega_0) = 0$. By Proposition \ref{UniSet}, $\td \times \partial_I S^{-1}(\Gamma)$ is a $C^\infty$-uniqueness set at $(\theta, 0)$ for all $\theta \in \td$. Therefore
\[f_1(\theta, I) = \oir, \hspace{1cm} g_2(\theta, I) = \oir. \]
Since $\overline{G}$ is symplectic it follows that
\[ J = D\overline{G}^T J D\overline{G} \hspace{0.5cm} \text{ where } \hspace{0.5cm} J = \left( \begin{array}{cc}
0 & I_d \\
-I_d & 0 \\
\end{array}
\right)_{2d \times 2d}, \]
and $I_d$ denotes the $d\times d$ identity matrix. A direct calculation yields to \[ 
J = \left( \begin{array}{cc}
0 & I_d + \partial_\theta g_1 \\
-I_d & -\partial_I g_1 + \partial_I g_1^T \\
\end{array}
\right) + \oir.
\]
Hence \[ \partial_\theta g_1(\theta, I) = \oir,\]
which implies 
\[ g_1(\theta, I) = [g_1](I) + \oir,\]
where $[g_1]$ denote the average of $g_1(\theta, I)$ over $\td$. Thus
\[ \overline{F}(\theta,I) = (\theta + \partial_I S(I), I) + \oir,\]
\[\overline{G}(\theta,I) = (\theta + [g_1](I), I) + \oir.\]
This shows that $\td \times \{ 0\}$ is a KAM torus for $\overline{G}$. Hence $T_{\omega_0} = \Sigma(\td \times \{ 0\})$ is a KAM torus for $G$. Furthermore 
\[ \overline{F} \circ \overline{G} = \overline{G} \circ \overline{F} + \oir\]
on a neighbourhood of $\td \times \{ 0\}$. By Lemma \ref{FlatAnalytic} $F$ and $G$ commute on a neighbourhood of $\td \times \{ 0\}$. To prove the last assertion let us suppose that $T_\omega$ is a KAM torus of $G$ with rotation vector $\omega$ for all $\omega \in \Gamma$. Then 
\[ [g_1](I) = \partial_I S(I) + \oir,\]
and 
\[ \overline{F} = \overline{G} + \oir.\]
Hence $F = G$ by Lemma \ref{FlatAnalytic}.
\end{proof}

 \begin{cor}
Let $F_0$, $\mathcal{U}_{\gamma, \tau}$ as in Theorem \ref{KAM_symplectic}. If $F,G \in \mathcal{U}_{\gamma, \tau}$ and $T_F, T_G$ coincide on a $C^\infty$-uniqueness set then $F=G$. 
\end{cor}

\subsection{Hamiltonian systems}
\label{sc: flows}

Let us state a simplified version of the KAM theorem for Hamiltonians systems found in \cite{poschel_integrability_1982}. 

\begin{thm}[KAM theorem for Hamiltonians]
\label{KAM_Hamiltonian}
Let $r, s > 0$ and $d \in \N$. Suppose $H_0 : \CBd{s} \rightarrow \C$ real analytic such that $\partial_I H_0\mid_{\CBd{s}}$ is a diffeomorphism onto its image and 
\[ \|\partial^2_IH_0\|_{\CBd{s}}, \|\partial^2_IH_0^{-1}\|_{\CBd{s}} < +\infty.\]
Let 
\[\Omega = \partial_I H_0(I).\]
Given $\gamma, \tau > 0$ there exists an open neighbourhood $\mathcal{U}_{\gamma, \tau}$ of $H_0$ in $C^\omega_\R(\CDom{r}{s}, \C)$ such that for any $H \in \mathcal{U}_{\gamma, \tau}$ there exist a Cantor set $K \subset \RBd{s}$ and Whitney smooth functions 
\[ h \in C^\infty(K, \R), \hone \Sigma \in \textup{Symp}^\infty(\T^d \times K, \T^d \times \R^d),\]
such that $\partial_I h : K \rightarrow \Omega^\tau_\gamma$ is a bijection and
\[ H \circ \Sigma \mid_{\td \times K}(\theta, I) = h(I), \hone X_{H \circ \Sigma} \mid_{\td \times C}(\theta, I) = (\partial_I h(I),0).\]
\end{thm}

Similar to the exact symplectic transformations case and using the notations in the previous theorem, it follows that for any $H \in \mathcal{U}_{\gamma, \tau}$ and for all $\omega \in \Omega^\tau_\gamma$ the graph of
\[ \Function{u_\omega}{\td}{\R^d}{\theta}{\Sigma (\theta, \partial_I h^{-1}(\omega))}\]
defines an invariant Lagrangian torus $T_\omega$ for $\Phi^t_H$ whose restricted dynamics is smoothly conjugated to a continuous translation by $\omega$. We encode the collection of invariant tori $\{T_{\omega}\}_{\omega \in \Omega^\tau_\gamma}$ in the \textit{KAM curve} $T_H$ defined as the Whitney smooth function
\[ \begin{array}{cccc} 
T_F : & \Omega^\tau_{\gamma} & \rightarrow & C^{\infty} (\mathbb{T}^d) \\ 
& \omega & \mapsto & u_\omega
\end{array}. \] 
As before, we say that a smooth Lagrangian invariant torus whose restricted dynamics are smoothly conjugated to a discrete translation by a Diophantine vector $\omega$ is a \textit{KAM torus} with \textit{rotation vector} $\omega$. The following is an analogous of Theorem \ref{main_symplectic_maps} in the Hamiltonian case.

\begin{thm}
\label{main_hamiltonians}
Let $H_0$, $\Omega$, $\mathcal{U}_{\gamma, \tau}$ as in Theorem \ref{KAM_Hamiltonian}. Let $H\in \mathcal{U}_{\gamma, \tau}$ and denote by $T_H$ the associated KAM curve. Suppose 
 \[ L  : \RDom{s} \rightarrow \R \] 
 is analytic and $\Gamma \subset \dioph$ is a $C^\infty$-uniqueness set at $\omega_0 \in \Gamma$ such that for all $\omega \in \Gamma$ the function $T_H(\omega)$ defines an invariant torus $T_\omega$ for the Hamiltonian flow $\Phi_L^t$. Then the following holds:
 \begin{enumerate}
 \item $T_{\omega_0}$ is a KAM torus for $L$ with rotation vector $\omega_0^L$ for some $\omega_0^L \in \R^d$. 
 \item If $\omega_0^L \parallel \omega_0$ then $X_L \parallel X_H$ on $T_{\omega_0}$.
 \item The flows $\Phi^t_H$ and $\Phi^t_L$ commute on a neighbourhood of $T_{\omega_0}$.
 \item If $T_\omega$ is a KAM torus for $L$ with rotation vector $\omega^L \parallel \omega$ for all $\omega \in \Gamma$ there exist an analytic function $\varphi$ such that $L = \varphi \circ F$ on an open neighbourhood of $T_{\omega_0}$.
\end{enumerate}
 \end{thm} 
 \begin{proof} Let $\Sigma, h$ as in Theorem \ref{KAM_Hamiltonian} when applied to $H$. We can assume WLOG that $\Sigma$, $h$ are well-defined smooth functions on open neighbourhoods of $\td \times K$ and $K$ respectively.  Denote
 \[ \mathcal{H} = H \circ \Sigma, \hspace{1cm} \mathcal{L} = L \circ \Sigma. \]
Let us suppose WLOG that $h^{-1}(\omega_0) = 0.$ Since $\td \times h^{-1}(\Gamma)$ is a $C^\infty$-uniqueness set at $(\theta,0)$ for all $\theta \in \td$ and it is invariant by the Hamiltonian flows $\Phi_{\mathcal{H}}^t$, $\Phi_{\mathcal{L}}^t$. 
Then 
\[ \partial_\theta \mathcal{H} = \oir, \hspace{1cm} \partial_\theta \mathcal{L} = \oir.\]
Hence
\[ \mathcal{H} = [\mathcal{H}](I) + \oir, \hspace{1cm} \mathcal{L}(\theta, I) = [\mathcal{L}](I) + \oir, \]
where $[\cdot]$ denote the average of the function over $\td$. By definition of $h$ 
\[ \partial_I [\mathcal{H}](I) = h(I) + \oir.\]
Thus
 \begin{equation}
\label{FlatFields}
\begin{array}{c}
 X_{\mathcal{H}}(\theta,I) = (h(I),0) + \oir,\\
 X_{\mathcal{L}}(\theta,I) = (\partial_I \mathcal{[L]}(I),0) + \oir.
 \end{array}
 \end{equation}
 In particular
 \begin{equation}
 \label{RotationVector} 
 X_{\mathcal{L}}(\theta, 0) = (\partial_I \mathcal{[L]}(0),0).
 \end{equation}
Hence $T_{\omega_0} = \Sigma(\td \times \{ 0\})$ is a KAM torus for $L$ with rotation vector 
\[\omega^L = \partial_I \mathcal{[L]}(0),\]
which proves the first assertion.  If $\omega^L \parallel \omega_0$ it follows from (\ref{FlatFields}) that $X_{\mathcal{H}} \parallel X_{\mathcal{L}}$ restricted to $T_{\omega_0}$. Thus by Proposition \ref{basics_Hamiltonians}, the vector fields $X_H, X_L$ are collinear on $T_{\omega_0}$. This shows the second assertion.  To prove the third one recall that the flows $\Phi^t_H,$ $\Phi^t_L$ commute if and only if the Poisson bracket $[X_H, X_L]$ vanishes. From (\ref{FlatFields})
\[ X_{\{ \mathcal{H}, \mathcal{L}\}} = \oir\]
and by Proposition \ref{basics_Hamiltonians}
 \[ [X_H, X_L] = X_{\{ H,L\}} = X_{\{ \mathcal{H}, \mathcal{L}\} \circ \Sigma^{-1}} = D\Phi \circ X_{\{ \mathcal{H}, \mathcal{L}\}} \circ \Sigma^{-1}. \]
By Lemma \ref{FlatAnalytic} it follows that 
\[ [X_H, X_L] = 0, \]
that is, $\Phi_{\mathcal{H}}^t$ and $\Phi_{\mathcal{L}}^t$ commute on a neighbourhood of $T_{\omega_0}$. To prove the last assertion let us suppose that $T_\omega$ is a KAM torus for $L$ with rotation vector $\omega^L \parallel \omega$ for all $\omega \in \Gamma$. Then there exist a smooth function $\gamma$ such that 
\[
 \partial_I[\mathcal{L}](I) = \gamma(I) h(I) + \oir. 
\]
Hence
\begin{equation}
\label{collinear} X_\mathcal{H} = \gamma X_\mathcal{L} + \oir.
 \end{equation}
Let us show that $X_H$ and $X_L$ are always collinear, that is
\[X_H = \langle X_H, X_L \rangle \dfrac{X_L }{\Vert X_L \Vert^2}. \]
By Lemma \ref{FlatAnalytic} it suffices to show that 
\[ X_H \circ \Sigma = \langle X_H \circ \Sigma, X_L \circ \Sigma \rangle \dfrac{X_L \circ \Sigma}{\Vert X_L \circ \Sigma \Vert^2} + \oir.\]
Developing the RHS and by (\ref{collinear})
\begin{align*}
\langle X_H \circ \Sigma, X_L \circ \Sigma \rangle \dfrac{X_L \circ \Sigma}{\Vert X_L \circ \Sigma \Vert^2} & = \langle D\Sigma \cdot X_\mathcal{H},D\Sigma \cdot X_\mathcal{L}\rangle \dfrac{D\Sigma \cdot X_\mathcal{L}}{\Vert D\Sigma \cdot X_\mathcal{L} \Vert^2} \\
& = \gamma D\Sigma \cdot X_\mathcal{L} + \oir \\
& = D\Sigma \cdot X_H + \oir\\
& = X_H \circ \Sigma + \oir
\end{align*}
Thus $X_H$, $X_L$ are everywhere collinear. In particular the level sets of $H$ and $L$ coincide. Let $p \in T_{\omega_0}$ and let us suppose WLOG that $H(p) = 0$. As $H$ is constant on every KAM torus it follows that $T_{\omega_0} \subset H^{-1}(0)$. By the implicit function theorem there exist an analytic diffeomorphism 
 \[ \Psi: (-\epsilon,\epsilon) \times U \subset \R \times \R^{d-1} \rightarrow W \subset \R^d\]
 such that 
 \[ H \circ \Psi (a,v) = a, \hspace{1cm} \Psi(0,0) = p.\]
Since $H$ and $L$ have the same level sets there exists $\varphi: (-\epsilon, \epsilon) \rightarrow \R$ analytic such that 
\[ L \circ \Psi (a,v) = \varphi(a).\]
Hence $L = \varphi \circ H$ on $W$, but clearly this equality holds also on the connected component of $T_{\omega_0}$ inside $H^{-1}(-\epsilon,\epsilon)$. This completes the proof.
\end{proof}

In the Hamiltonian case a little more can be said for systems sharing a sufficiently big collection of tori even if the rotation vectors on these tori are not collinear. 

\begin{thm}
\label{SimConjug}
Let $H_0$, $\Omega$, $\mathcal{U}_{\gamma, \tau}$ as in Theorem \ref{KAM_Hamiltonian}. Let $L_1 \in \mathcal{U}_{\gamma, \tau}$ and denote by $T_{L_1}$ its associated KAM curve. Suppose 
\[L_2,\dots,L_{d}: \RDom{s} \rightarrow \R\]
are analytic and $\Gamma \subset \dioph$ is a $C^\infty$-uniqueness set at $\omega_0 \in \Gamma$ such that for all $\omega \in \Gamma$ the function $T_{L_1}(\omega)$ defines an invariant torus $T_\omega$ for the Hamiltonian flow $\Phi^t_{L_i}$  and any $i \in \{1, \dots, d\}$. Denote by $\omega^{i}$ the rotation vector of $T_{\omega_0}$ under the flow $\Phi^t_{L_i}$ (see Theorem \ref{main_hamiltonians}). If \[\omega^1, \omega^2, \dots, \omega^{d} \text{ are linearly independent,}\]
then there exist a symplectic change of coordinates $\psi$ conjugating $L_1, \dots , L_{d}$ simultaneously to completely integrable Hamiltonians in a neighbourhood of $T_{\omega_0}$.
\end{thm}
 
\begin{proof} In the following we will show the existence of an analytic change of coordinates $\psi$, not necessarily symplectic, conjugating simultaneously $L_1, \dots, L_d$ to integrable Hamiltonians. We will discuss in Section \ref{sc: upgrade_conjugacy} how this argument can be adapted to define $\psi$ to be symplectic. This will rely on a `Darboux lemma' for Lagragian foliations (Proposition \ref{LagFol}).

Let $\Sigma, h$ as in Theorem \ref{KAM_Hamiltonian} when applied to $L_1$. We can assume WLOG that $\Sigma$, $h$ are well-defined smooth functions on open neighbourhoods of $\td \times K$ and $K$ respectively.  Denote $\mathbf{L} = (L_1,\dots, L_{d}),$ and define
\[ \Function{\Psi}{\RDom{s}}{\TR}{(\theta, I)}{ (\theta,\mathbf{L}(\theta, I))}.\]
Since the function $\mathbf{L}$ is constant on every common invariant torus, for every $\omega \in \Gamma$ there exist an unique vector $\mathfrak{h}(\omega) \in \R^d$ such that
\[ \Psi(T_\omega) = \td \times \{ \mathfrak{h}(\omega) \}.\]
An explicit formula for $\mathfrak{h}$ can be retrieved by means of the function $h$. Indeed 
\[ \mathfrak{h}(\omega) = [\mathbf{L}\circ \Sigma](h^{-1}(\omega)).\]
In particular, $\mathfrak{h}$ is a Whitney smooth function. Suppose for a moment that $\Psi$, when restricted to a sufficiently small neighbourhood of $\td \times \{0\}$, is a symplectomorphism onto its image and denote 
\[ \mathfrak{L}_i = L_i \circ \Psi^{-1}\]
for all $i = 1,\dots,d$. By Proposition \ref{basics_Hamiltonians} the flow generated by the vector field $X_{L_i}$ associated to $(L_i,\omega)$ is equivalent to the flow generated by the vector field $Y_{\mathfrak{L}_i}$ associated to $(\mathfrak{L}_i, (\Psi^{-1})^\ast(\omega)).$ By the invariance of $\td \times \mathfrak{h}(\Gamma)$ under the flow given by $Y_{\mathfrak{L}_i}$ it follows that 
\begin{equation}
\label{eq: Y_invariance}
Y_{\mathfrak{L}_i}(\theta, I) = (Y_i(\theta, I),0)
\end{equation}
for all $(\theta, I) \in \td \times \mathfrak{h}(\Gamma)$ and for some analytic function $Y_i$. By Proposition \ref{UniSet}, $\td \times \mathfrak{h}(\Gamma)$ is a $C^\infty$-uniqueness set at $(\theta, \mathfrak{h}(\omega_0))$ for all $\theta \in \td$. Since $Y_i$ is analytic, it follows that (\ref{eq: Y_invariance}) holds for all $(\theta, I)$, which shows the integrability of the Hamiltonian $\mathfrak{L}_i$. Hence, supposing that $\Psi$ restricted to a sufficiently small neighbourhood of $\td \times \{0\}$, the theorem follows by taking $\psi = \Psi^{-1}$. 

Therefore it suffices to show that $\Psi$ restricted to a sufficiently small neighbourhood of $T_{\omega_0}$ is a diffeomorphism onto its image. To simplify the notation let us suppose that $h^{-1}(\omega_0) = 0$. Since $T_{\omega_0} = \Sigma(\td \times \{0\})$ it suffices find a neighbourhood $U$ of $\td \times \{ 0\}$ such that $\Psi \circ \Sigma \mid_U$ is a diffeomorphism onto its image. Denote 
\[ \mathcal{L}_i = L_i \circ \Sigma\]
for all $i = 1,\dots,d$. As the pair $L_1,L_i$ satisfy the hypotheses of Theorem \ref{main_hamiltonians} for $i = 1,\dots,d,$ equation (\ref{RotationVector}) holds replacing $L$ by $L_i$ which yields to 
\begin{equation}
\label{FlatVectorField}
X_{\mathcal{L}_i}(\theta,0) = (\partial_I\mathcal{L}_i(\theta,0), -\partial_\theta\mathcal{L}_i(\theta,0)) = (\omega^i,0) 
\end{equation}
for all $\theta \in \td$. Thus 
\[ D(\Psi \circ \Sigma) (\theta, 0) = \left( \begin{array}{cc}
I_{d} & 0 \\
\ast & \begin{array}{cccc} \omega^1 \\ \omega^2 \\ \dots \\ \omega^{d} \end{array} \\
\end {array} \right)\]
for all $\theta \in \td$. By hypotheses $\omega^1, \omega^2, \dots, \omega^{d}$ are linearly independent which shows that $\Psi \circ \Sigma$ is a local diffeomorphism on a small neighbourhood $U$ of $\td \times \{ 0\}$. Since $\Psi \circ \Sigma$ restricted to $\td \times \{ 0\}$ is injective we can suppose, up to consider a smaller neighbourhood, that $\Psi \mid_U$ is a diffeomorphism onto its image. This concludes the proof. \end{proof}

For $d = 2$, Theorems \ref{main_hamiltonians} and \ref{SimConjug} imply the following.

\begin{thm}
For $d = 2$ and under the same hypotheses of Theorem \ref{main_hamiltonians} we have the following dichotomy: 
\begin{enumerate}
\item If $\omega^L \nparallel \omega_0$ there exist a symplectomorphism conjugating $H$ and $L$ simultaneously to completely integrable systems in a neighbourhood of $T_{\omega_0}$.
\item If $\omega^L \parallel \omega_0$ the vector fields $X_H, X_L$ are collinear on $T_{\omega_0}$. Furthermore, either \[ K = \varphi \circ H\] in a neighbourhood of $T_{\omega_0}$ for some analytic function $\varphi$ or every neighbourhood of $T_{\omega_0}$ contains open connected sets completely foliated by common invariant tori of $H$ and $L$. 
\end{enumerate}
\end{thm}

\begin{cor}
Let $H_0$, $\mathcal{U}_{\gamma, \tau}$ as in Theorem \ref{KAM_Hamiltonian}. Suppose $H, L \in \mathcal{U}_{\gamma, \tau}$. If $T_H = T_L$ on a $C^\infty$-uniqueness set then $H = K$.
\end{cor}

\subsection{Symplectic conjugacy in Theorem \ref{SimConjug}}
\label{sc: upgrade_conjugacy}

Let us discuss how to modify the proof of Theorem \ref{SimConjug} so that the diffeomorphism conjugating simultaneously the Hamiltonians is a symplectomorphism. First we adapt a result of H. Eliasson \cite[Proposition 3]{eliasson_normal_1990}, a `Darboux Lemma' preserving a Lagrangian foliation. 

\begin{prop}
\label{LagFol}
Let $U \subset \R^d$ be an open neighbourhood of $0$. Suppose $\omega$ is an exact symplectic form on $\td \times U$ such that the foliation $\mathcal{F} = \lbrace \T^d \times \{ I\}\rbrace_{I \in U}$ is Lagrangian and 
\begin{equation}
\label{LagFolHyp}
 \omega\mid_{ \td \times \{ 0 \}} = \omega_{\std}\mid_{\td \times \{ 0 \}}
 \end{equation}
 where $\omega_{\std}$ denotes the canonical symplectic form in $\td \times U$.  Then there exist an open neighbourhood $V \subset U$ of $\td \times \{0\}$ and an analytic diffeomorphism $\phi: \td \times V \rightarrow \td \times V$  preserving the foliation $\mathcal{F}$ and such that \[ \phi ^{\ast}\omega = \omega_{\std}.\]
\end{prop}

\begin{proof} We will define the diffeomorphism $\phi$ by applying Moser's trick.  Define 
\[ \omega _t = \omega_{\std} + t(\omega - \omega_{\std}).\]
 Notice that $\omega_0 = \omega_{\std}$ and that $\omega _t$ is symplectic for all $t \in [0,1]$ in some open neighbourhood $U$ of $\td \times \lbrace 0 \rbrace$. We may assume WLOG this neighbourhood to be $U$. Let $\beta \in \Omega^1(\td \times V)$ be such that $d\beta = \omega - \omega_{0}$ and define $X_t$ as the unique time dependent vector field $X_t$ obeying
\begin{equation}
\label{beta_equation}
i_{X_t} \omega_t = \beta.
\end{equation}
Let $\alpha, \alpha _0$ be primitives of $\omega, \omega_{0}$ respectively.  We will define $\beta$ of the form 
\begin{equation}
\label{beta_form}
\beta = \alpha - \alpha_0 - df,
\end{equation}
where $f$ is a real analytic function on $\td$, so that the time one map $\phi$ associated to the vector field $X_t$ verifies the proposition. Let us denote by $\phi_t$ the flow associated to $X_t$. Recall that by Cartan's magic formula 
\[ \frac{d}{dt}(\phi_t ^*\omega_t )= 0.\]
Thus, assuming $\phi = \phi_1$ is well defined, we have
\[\phi^*\omega = \phi_1^*\omega_1  = \phi_0^*\omega_{0} = \omega_\std.\] 
Therefore, to prove the proposition it suffices to define $f$ (and hence $\beta$) so that the flow $\phi_t$ preserves the foliation $\mathcal{F}$.

Denote by $X_i = \frac{\partial}{\partial \theta _i}$ the coordinate vector fields. Then  $\phi_t$ preserves the foliation $\mathcal{F}$ if and only if $X_t (dI_i) $ is independent of $\theta$ for every $i =1,\dots,d$. By  (\ref{beta_equation}), the last condition is satisfied if and only if $\beta$ depends only on $I$. In other words, the flow $\phi_t$ preserves $\mathcal{F}$ if and only if  $\beta (X_i)  \text{ is constant for every } i = 1, \dots, d.$ 

Denote $g_i = (\alpha - \alpha_0) (X_i)$ and $g = (g_1,\dots,g_d)$. Since we suppose $\beta$ is of the form (\ref{beta_form}) it follows that the flow $\phi^t$ preserves $\mathcal{F}$ if and only if  
\begin{equation}
\label{beta_constant}
g - \nabla f = constant.
\end{equation}
 The foliation $\mathcal{F}$ being Lagrangian for $\omega$ is equivalent to 
\[ \omega (X_i,X_j) = 0.\] 
Since $[X_i,X_j] = 0$ this yields to
\[ X_j(\alpha (X_i)) = X_i(\alpha(X_j)).\]
Note that the previous equations also hold if we replace $\omega$ and $\alpha$ by $\omega _0$ and $\alpha _0$ respectively. Thus 
\[ \dfrac{\partial g_i}{\partial \theta _j} = \dfrac{\partial g_j}{\partial \theta _i}.\]
Hence there exist a function $h: \R ^d \rightarrow \R$ and constants $b, c \in \R^d$ such that 
\[ \widetilde{g} = \nabla h + b, \hskip1cm h(x + e_i) = h(x) + c_i,\]
where $\widetilde{g}$ is the lift of $g$ to $\R ^d$ and $e_1, \dots, e_d$ denote the canonic base in $\R^d$. Therefore $\widetilde{f}: \R^d \rightarrow \R$ given by \[ \widetilde{f}(x) = h(x) - \langle c , x \rangle \] is the lift of a well-defined on $\T^d$. Define $f$ as the function induced by $\widetilde{f}$ and notice that in this case (\ref{beta_constant}) is satisfied. Hence the flow $\phi_t$ preserves the $\mathcal{F}$ and the proposition follows. 
\end{proof}

We would like to apply the last result to the symplectic form  $\omega = \psi^*\omega_{\std}$ where $\psi$ is the analytic diffeomorphism obtained in the proof of Theorem \ref{SimConjug}. In that case, $\psi \circ \phi$ would be the desired symplectic map. Nevertheless (\ref{LagFolHyp}) does not necessarily holds for $\psi^*\omega_{\std}$ and thus Proposition \ref{LagFol} cannot be applied. We will fix this by modifying the diffeomorphism $\psi$ as follows.

\begin{proof}[Addendum to the proof of Theorem \ref{SimConjug}]
Let $\Sigma, h$ as in Theorem \ref{KAM_Hamiltonian} when applied to $L_1$. We can assume WLOG that $\Sigma$, $h$ are well-defined smooth functions on open neighbourhoods of $\td \times K$ and $K$ respectively and that $h^{-1}(\omega_0) = 0$. We denote $\Sigma = (\Sigma_1, \Sigma_2)$ and define $g : \td \rightarrow \td$, $\gamma: \td \rightarrow \R^d$ by 
\[ g(\theta) := \Sigma_1(\theta,0), \hspace{1cm} \gamma(\theta) := \Sigma_2(\Sigma_1^{-1}(\theta,0)).\]
Notice that $g \in \textup{Diff}^\omega(\td)$ because $T_{\omega_0} = \Sigma (\td \times \{ 0 \})$ is a Lagrangian graph and thus $\gamma \in C^\omega(\td, \rd)$ is well defined. Let $\phi_1, \phi_2 : \TR \rightarrow \TR$ 
\[ \phi_1(\theta, I) = (\theta, \gamma(\theta) + I), \hspace{1cm} \phi_2(\theta, I) =\left(g(\theta), \partial_\theta g^{-1}(\theta)^T I\right).\]
 These two mappings are symplectic with respect to $\omega_{\std}$, where $\omega_{\std}$ denotes the standard symplectic form in $\TR$ (see \cite[Lemma 1.2.4]{bost_tores_1986}). Notice that $T_{\omega_0} = (\phi_1 \circ \phi_2) (\td \times \{ 0 \}).$  Denote 
 \[ \mathcal{L}_i = L_i \circ \Sigma, \hskip1cm \overline{L}_i := L_i \circ \phi_1 \circ \phi_2,\]
 for all $i = 1, \cdots, d$. Let us recall equation (\ref{FlatVectorField})
 \[X_{\mathcal{L}_i}(\theta,0) = (\partial_I\mathcal{L}_i(\theta,0), -\partial_\theta\mathcal{L}_i(\theta,0)) = (\omega^i,0),\]
 which we obtained previously in the proof of Theorem \ref{SimConjug}. It follows that for all $\theta \in \td$ and all $t \in \R$ 
\begin{align*}
(\theta + t\omega_i, 0) & = \Phi^t_{\mathcal{L}_i}(\theta, 0) \\
& = \Phi^t_{L_i \circ \phi_1 \circ \phi_1^{-1} \circ \Sigma}(\theta, 0) \\
& = \left(\Sigma^{-1} \circ \phi_1 \circ \Phi^t_{L_i \circ \phi_1} \circ \phi_1^{-1} \circ \Sigma\right) (\theta, 0) \\
& = \left( \phi_2^{-1} \circ \Phi^t_{L_i \circ \phi_1} \circ \phi_2 \right) (\theta, 0) \\
& = \Phi^t_{\overline{L}_i}(\theta,0).
\end{align*}
Thus (\ref{FlatVectorField}) holds if we replace $\mathcal{L}_i$ by $\overline{L}_i$, namely
\begin{equation}
\label{eq: L_derivative}
X_{\overline{L}_i}(\theta, 0) = (\partial_I \overline{L}_i (\theta, 0), -\partial_\theta \overline{L}_i (\theta, 0)) =  (\omega^i, 0).
\end{equation}
Let \[ \overline{\mathbf{L}} = (\overline{L}_1,\dots, \overline{L}_{d}), \hspace{1cm} M = \left( \begin{array}{cccc} \omega^1 \\ \omega^2 \\ \dots \\ \omega^{d} \end{array} \ \right)^{-1},\] 
and define
\[ \Function{\overline{\Psi}}{\td \times V}{\td \times \rd}{(\theta, I)}{(\theta, M\overline{\mathbf{L}}(\theta, I))}\]
where $V$ is a neighbourhood of $0$ such that $\phi_1 \circ \phi_2 (V) \subset I$. By (\ref{eq: L_derivative}) 
\[ D\overline{\Psi} (\theta, 0) = I_{2d}\]
 for all $\theta \in \td$. Hence, $\overline{\Psi}$ is a local diffeomorphism on a neighbourhood $U$ of $\td \times \{ 0 \}$ and, up to consider a smaller neighbourhood, we can suppose WLOG that $\overline{\Psi}\mid_U$ is a diffeomorphism to its image. As before, we can show that for all $i = 1, \dots, d$ the Hamiltonian flow associated to $\overline{L}_i \circ \Psi^{-1}$ preserves all the tori of the form $\td \times \{ I\}$ contained in $\overline{\Psi}(U)$. Since the Hamiltonian is analytic this implies the complete integrability of the system. Notice that the symplectic form 
\[\omega = \left(\overline{\Psi}^{-1}\right)^*\omega_{\std}\]
satisfies (\ref{LagFolHyp}). Hence by Proposition \ref{LagFol} there exists an analytic diffeomorphism $\phi$, preserving all the tori of the form $\td \times \{I\}$, such that $\phi^*\omega = \omega_{\std}$. Thus $\psi = \overline{\Psi}^{-1} \circ \phi$ is the desired symplectic transformation.
\end{proof}

\bibliographystyle{AIMS.bst}
\bibliography{analytic.bib}

\providecommand{\href}[2]{#2}
\providecommand{\arxiv}[1]{\href{http://arxiv.org/abs/#1}{arXiv:#1}}
\providecommand{\url}[1]{\texttt{#1}}
\providecommand{\urlprefix}{URL }
\begin{thebibliography}{1}

\bibitem{arnold_mathematical_2007}
\newblock V.~I. Arnold, V.~V. Kozlov and A.~I. Neishtadt,
\newblock \emph{Mathematical aspects of classical and celestial mechanics},
  vol.~3,
\newblock Springer Science \& Business Media, 2007.

\bibitem{bost_tores_1986}
\newblock J.-B.~t. Bost,
\newblock Tores invariants des syst{\`e}mes dynamiques hamiltoniens
  (d'apr{\`e}s {Kolmogorov}, {Arnol}'d, {Moser}, {R{\"u}ssmann}, {Zehnder},
  {Herman}, {P{\"o}schel}),
\newblock \emph{Ast{\'e}risque}, 113--157,
\newblock
  \urlprefix\url{https://mathscinet.ams.org/mathscinet-getitem?mr=837218}.

\bibitem{cannas_da_silva_lectures_2001}
\newblock A.~Cannas~da Silva,
\newblock \emph{Lectures on symplectic geometry}, vol. 1764 of Lecture {Notes}
  in {Mathematics},
\newblock Springer-Verlag, Berlin, 2001,
\newblock
  \urlprefix\url{https://mathscinet.ams.org/mathscinet-getitem?mr=1853077}.

\bibitem{carminati_there_2014}
\newblock C.~Carminati, S.~Marmi and D.~Sauzin,
\newblock There is only one {KAM} curve,
\newblock \emph{Nonlinearity}, \textbf{27} (2014), 2035,
\newblock \urlprefix\url{http://stacks.iop.org/0951-7715/27/i=9/a=2035}.

\bibitem{eliasson_normal_1990}
\newblock L.~H. Eliasson,
\newblock Normal forms for {Hamiltonian} systems with {Poisson} commuting
  integrals{\textemdash}elliptic case,
\newblock \emph{Commentarii Mathematici Helvetici}, \textbf{65} (1990), 4--35,
\newblock \urlprefix\url{https://doi.org/10.1007/BF02566590}.

\bibitem{eliasson_around_2015}
\newblock L.~H. Eliasson, B.~Fayad and R.~Krikorian,
\newblock Around the stability of {KAM} tori,
\newblock \emph{Duke Mathematical Journal}, \textbf{164} (2015), 1733--1775,
\newblock \urlprefix\url{http://projecteuclid.org/euclid.dmj/1434377460}.

\bibitem{lazutkin_existence_1973}
\newblock V.~F. Lazutkin,
\newblock Existence of caustics for the billiard problem in a convex domain,
\newblock \emph{Izvestiya Akademii Nauk SSSR. Seriya Matematicheskaya},
  \textbf{37} (1973), 186--216,
\newblock
  \urlprefix\url{https://mathscinet.ams.org/mathscinet-getitem?mr=0328219}.

\bibitem{poschel_integrability_1982}
\newblock J.~P{\"o}schel,
\newblock Integrability of {Hamiltonian} systems on {Cantor} sets,
\newblock \emph{Communications on Pure and Applied Mathematics}, \textbf{35}
  (1982), 653--696,
\newblock
  \urlprefix\url{https://mathscinet.ams.org/mathscinet-getitem?mr=668410}.

\bibitem{shang_note_2000}
\newblock Z.-j. Shang,
\newblock A {Note} on the {KAM} {Theorem} for {Symplectic} {Mappings},
\newblock \emph{Journal of Dynamics and Differential Equations}, \textbf{12}
  (2000), 357--383,
\newblock \urlprefix\url{https://doi.org/10.1023/A:1009068425415}.

\end{thebibliography}

\end{document}